\documentclass[12pt,reqno,a4paper]{amsart}
\usepackage{hyperref}
\usepackage{enumerate}
\usepackage[left=0.65in, right=0.65in, top= 1in, bottom=1.65in]{geometry}
\usepackage{amsmath,amssymb,amsthm,amsfonts}
\usepackage{graphicx}
\usepackage{tikz}
\usetikzlibrary{graphs, graphs.standard}

\tikzset{
	modal/.style={>=stealth’,shorten >=1pt,shorten <=1pt,auto,node distance=1.5cm,
		semithick},
	world/.style={circle, draw,minimum size=.1cm,fill=gray!15},
	point/.style={circle,draw,inner sep=0.3mm,fill=black},
	circ/.style={circle,draw,inner sep=0.1mm,fill=white},
	reflexive above/.style={->,loop,looseness=7,in=120,out=60},
	reflexive below/.style={->,loop,looseness=7,in=240,out=300},
	reflexive left/.style={->,loop,looseness=7,in=150,out=210},
	reflexive right/.style={->,loop,looseness=7,in=30,out=330}
}

\usetikzlibrary{shapes}
\usetikzlibrary{plotmarks}
\usetikzlibrary{arrows}
\usetikzlibrary{positioning}
\theoremstyle{definition}
\newtheorem{defn}{Definition}[section]
\newtheorem{fact}[defn]{Fact}
\newtheorem{prop}[defn]{Proposition}
\newtheorem{thm}[defn]{Theorem}

\newtheorem{lem}[defn]{Lemma}
\newtheorem{question}[defn]{Question}
\newtheorem{remark}[defn]{Remark}
\newtheorem{observation}[defn]{Observation}
\newtheorem{claim}[defn]{Claim}
\title[list-distinguishing chromatic numbers]{Upper bounds for the list-distinguishing chromatic number}

\author{Amitayu Banerjee$^{\ast}$}
\address{Alfr\'ed R\'enyi Institute of Mathematics, Reáltanoda utca 13-15, 1053, Budapest, Hungary}
\email[Corresponding author]{banerjee.amitayu@gmail.com}
\thanks{$^{\ast}$ Corresponding author.}

\author{Zal\'{a}n Moln\'{a}r}
\address{E\"otv\"os Lor\'and University, Department of Logic, M\'{u}zeum krt. 4, 1088, Budapest, Hungary}
\email{mozaag@gmail.com}

\author{Alexa Gopaulsingh}
\address{E\"otv\"os Lor\'and University, Department of Logic, M\'{u}zeum krt. 4, 1088, Budapest, Hungary}
\email{alexa279e@gmail.com}

\date{}

\makeatletter
\@namedef{subjclassname@2020}{\textup{2020} Mathematics Subject Classification}
\makeatother
\begin{document}
\begin{abstract}
In this paper, all results apply only to finite graphs.
Let $G$ be a simple connected finite graph with $n$ vertices and maximum degree $\Delta(G)$.
We show that the list-distinguishing chromatic number $\chi_{D_{L}}(G)$ of $G$ is at most $2\Delta(G)-1$, and it is $2\Delta(G)-1$ if $G$ is a complete bipartite graph $K_{\Delta(G),\Delta(G)}$ or a cycle with six vertices.
We apply a result of Lov\'{a}sz to reduce the above-mentioned upper bound of $\chi_{D_{L}}(G)$ for certain graphs. 
We also show that if $H$ is a connected unicyclic graph of girth of at least seven and $\Delta(H)\geq 3$, then $\chi_{D_{L}}(H)$ is at most $\Delta(H)$.
Moreover, we obtain two upper bounds for $\chi_{D_{L}}(G)$ in terms of the coloring number of $G$ and the list chromatic number of $G$.
We also determine the list-distinguishing chromatic number for some special graphs.
\end{abstract}

\maketitle
\section{Introduction}
All graphs considered in this paper are simple, finite, and connected, and all our results apply only to finite graphs. In \cite{AC1996}, Albertson and Collins studied the distinguishing number of a graph, and Collins and Trenk \cite{CT2006} introduced the distinguishing chromatic number of a graph. 
A coloring $h:V_{G}\rightarrow\{1,...,r\}$ of the vertices of a graph $G=(V_{G}, E_{G})$ is {\em $r$-distinguishing} provided no nontrivial automorphism of $G$ preserves all of the colors of the vertices. The {\em distinguishing number} of $G$, denoted by $D(G)$, is the minimum integer $r$ such that $G$ has an $r$-distinguishing coloring. The coloring $h$ is {\em $r$-proper distinguishing} provided $h$ is $r$-distinguishing and a proper coloring. The {\em distinguishing chromatic number} of $G$, denoted by $\chi_{D}(G)$, is the minimum integer $r$ such that $G$ has an $r$-proper distinguishing coloring.
Collins and Trenk \cite{CT2006} obtained a Brooks' Theorem type upper bound for $\chi_{D}(G)$.

\begin{thm}\label{Theorem 1.1}{(Collins and Trenk \cite[Theorem 4.5]{CT2006})}
{\em Let $G$ be a connected graph. Then $\chi_{D}(G)\leq 2\Delta(G)-1$ unless $G$ is either the complete bipartite graph $K_{\Delta(G),\Delta(G)}$, or the cycle graph $C_{6}$. In these cases, $\chi_{D}(G) = 2\Delta(G)$}.    
\end{thm}

Recently, Ferrara et al. \cite{FGHSW2013} extended the notion of distinguishing proper coloring to a list distinguishing proper coloring. Given an assignment $L = {L(v)}_{v\in V_{G}}$ of lists of available colors to the vertices of $G$, we say that $G$ is {\em (properly) L-distinguishable} if there is a (proper) distinguishing coloring $f$ of $G$ such that $f(v) \in L(v)$ for all $v\in V_{G}$. The {\em list-distinguishing number}
of $G$, denoted by $D_{L}(G)$, is the minimum integer $k$ such that $G$ is $L$-distinguishable for any list assignment $L$ with $\vert L(v)\vert = k$ for all $v\in V_{G}$. Similarly, the {\em list-distinguishing chromatic number} of $G$, denoted by $\chi_{D_{L}}(G)$, is the minimum integer $k$ such that $G$ is properly $L$-distinguishable for any list
assignment $L$ with $\vert L(v)\vert = k$ for all $v\in V_{G}$. If $\chi(G)$ is the chromatic number of $G$, then the following holds:

\begin{center}
$\chi(G)\leq max\{\chi(G), D(G)\} \leq \chi_{D}(G) \leq \chi_{D_{L}}(G)$.
\end{center}

The second inequality was stated in \cite{CT2006}. Since all lists can be identical, the third inequality holds.

\subsection{Motivation}
Our next example shows that $\chi_{D_L}(G)\neq \chi_D(G)$ in general. This motivates the study of list-distinguishing chromatic numbers.
Consider the graph $G$ given in Figure \ref{Figure 1}. Since $G$ is bipartite and asymmetric, we have $\chi_D(G) = 2$. Consider the following assignment of lists: $L(v_0) =\{1,2\}$, $L(v_1) = \{1,3\}$, $L(v_2) = \{2,3\}$, $L(v_3) = \{1,4\}$, $L(v_4) = \{2,4\}$, $L(v_5) =\{3,4\}$, and for $0\leq i \leq 6$, $L(s_i)$ can be an arbitrary set of two elements. Then there is no proper coloring of $G$ from $\{L(v)\}_{v\in V_G}$, since the set $\{v_0,\dots, v_5\}$ does not have a proper coloring from the given lists. Thus, $\chi_{D_L}(G) \neq 2$. Moreover, $\chi_{D_L}(G) = 3$.

\begin{figure}[!ht]
\begin{minipage}{\textwidth}
\centering
\begin{tikzpicture}[scale=0.7]

\draw[black,] (0,0) -- (-3,-2); 
\draw[black,] (0,0) -- (-1,-2);
\draw[black,] (0,0) -- (1,-2);
\draw[black,] (0,0) -- (3,-2); \draw[black,] (3,-2) -- (4,-2); \draw[black,] (4,-2) -- (5,-2); 
\draw[black,] (-1,-2) -- (-2,-2); 

\draw[black,] (0,-4) -- (-3,-2);
\draw[black,] (0,-4) -- (-1,-2);
\draw[black,] (0,-4) -- (1,-2);
\draw[black,] (0,-4) -- (3,-2);
\draw[black,] (0,-4) -- (0,-5);

\draw[black,] (-6,-2) -- (-5,-2);
\draw[black,] (-5,-2) -- (-4,-2);
\draw[black,] (-4,-2) -- (-3,-2);

\node[circ] at (0,0) {$v_0$};

\node[circ] at (-3,-2) {$v_1$};
\node[circ] at (-4,-2) {$s_{2}$};
\node[circ] at (-5,-2) {$s_{1}$};
\node[circ] at (-6,-2) {$s_{0}$};

\node[circ] at (-1,-2) {$v_2$};
\node[circ] at (-2,-2) {$s_3$};

\node[circ] at (1,-2) {$v_3$};
\node[circ] at (3,-2) {$v_4$};
\node[circ] at (4,-2) {$s_4$};
\node[circ] at (5,-2) {$s_5$};

\node[circ] at (0,-4) {$v_5$};
\node[circ] at (0,-5) {$s_6$};
\end{tikzpicture}    

\end{minipage}
\caption{\em Graph $G$ where $\chi_{D_L}(G)\neq \chi_D(G)$}
\label{Figure 1}
\end{figure}
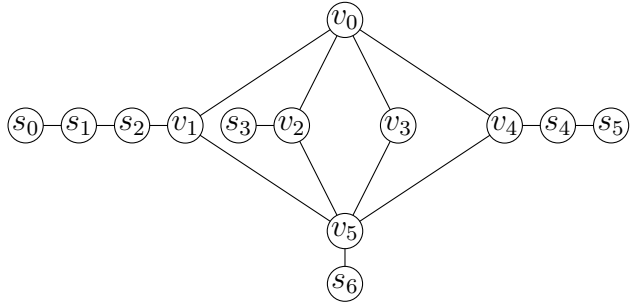

\subsection{Structure of the paper}
In Section 3, inspired by Theorem \ref{Theorem 1.1}, the first two authors prove that if $G$ is connected, then $\chi_{D_L}(G)\leq 2\Delta(G)-1$, and the equality holds if $G$ is $K_{\Delta(G),\Delta(G)}$ or $C_{6}$. In view of the graph in Figure \ref{Figure 1}, this strengthens Theorem \ref{Theorem 1.1}.
We remark that $\chi_{D_L}(G)\leq \Delta(G)$ if $G$ is a connected unicyclic graph of girth at least 7 and $\Delta(G)\geq 3$ (see Theorem \ref{Theorem 3.1}, Remark \ref{Remark 3.2}).
    
\qquad \qquad Borodin and Kostochka \cite{BK1977}, Catlin \cite{Cat1978}, and Lawrence \cite{Law1978} independently improved the Brooks-upper bounds for $\chi (G)$ for graphs omitting small cliques by applying a result of Lov\'{a}sz from \cite{Lov1966}.
In particular, they proved that if $K_{r} \not\subseteq G$, where $4 \leq r \leq \Delta(G)+1$, then $\chi(G) \leq \frac{r-1}{r}(\Delta(G) + 2)$.
In Section 4, the first two authors work in a similar fashion to reduce the upper bound of $\chi_{D_{L}}(G)$ from Theorem \ref{Theorem 3.1} for certain graphs by applying the same result of Lov\'{a}sz. In particular, we show that if $G$ is $(n-r+1)$-connected and does not contain a complete bipartite graph $K_{r-1,r-1}$ where $7\leq r\leq \Delta(G)+1$, then $\chi_{D_{L}}(G)\leq 2\Delta(G)- (3\lfloor\frac{(\Delta(G)+1)}{r}\rfloor-2)$ (see Theorem \ref{Theorem 4.3}).

\qquad \qquad In Section 5, we prove two upper bounds for $\chi_{D_{L}}(G)$ in terms of the coloring number $Col(G)$ and the list chromatic number $\chi_{L}(G)$ (see Theorem \ref{Theorem 5.1}):
\begin{enumerate}
    \item $\chi_{D_{L}}(G)\leq Col(G)D_{L}(G)$.
    
    \item If $Aut(G)\cong \Sigma$ where $\Sigma$ is a finite abelian group so that $Aut(G)\cong \prod_{1\leq i\leq k} \mathbb{Z}_{p_{i}^{n_{i}}}$ for some $k$, where $p_{1},...,p_{k}$ are primes not necessarily distinct, then $\chi_{D_{L}}(G)\leq \chi_{L}(G)+k$.
\end{enumerate}

In Theorem \ref{Theorem 5.1}, we also give new examples to show that the result mentioned in (2) is sharp.
We also determine the list-distinguishing chromatic number of the book graphs.

\section{Connected graphs with maximum degree two}

Let $P_{n}$ be a path of $n$
vertices, $C_{n}$ be a cycle of $n$ vertices, $K_{n}$ be a complete graph of $n$ vertices, and $K_{n,m}$ be a complete bipartite graph with bipartitions of size $n$ and $m$.
\begin{fact}\label{Fact:2.1}
{\em The following holds:
\begin{enumerate}
    \item If $T$ is a tree, then $\chi_{D}(T) = \chi_{D_{L}}(T)$ (see \cite[Theorem 7]{FGHSW2013}).
    \item If $T$ is a tree, then $\chi_{D}(T) \leq \Delta(T) + 1$ (see \cite[Theorem 3.4]{CT2006}).
    
    \item If $G$ is a graph, then $\chi_{D}(G) \leq \chi_{D_{L}}(G)$.
    
    \item $\chi_D(P_{2t})=2 $, $\chi_D(P_{2t+1})=3 $, $\chi_D(C_{4})=4 $, $\chi_D(C_{5})=3 $, $\chi_D(C_{6})=4 $, $\chi_D(C_{2n})=3 $, $\chi_D(C_{2m+1})=3 $ for any $t\geq 1$, $n\geq 4 $, and $m\geq 3$ (see \cite[Theorem 2.2]{CT2006}). 
\end{enumerate}
}
\end{fact}

\begin{prop}\label{Proposition 2.2}
{\em $\chi_{D_{L}}(P_{2t})=2$, $\chi_{D_{L}}(P_{2t+1})=3$, and $\chi_{D_{L}}(C_{4})=4$ for any $t\geq 1$.
}
\end{prop}

\begin{proof}
By Fact \ref{Fact:2.1}(1,4), we have $\chi_{D_{L}}(P_{2t})=\chi_{D}(P_{2t})=2$ and $\chi_{D_{L}}(P_{2t+1})=\chi_{D}(P_{2t+1})=3$.
Since $4 = \chi_D(C_4)\leq \chi_{D_L}(C_4)\leq 4$, we have $\chi_{D_L}(C_4)= 4$.
\end{proof}

\begin{prop}\label{Proposition 2.3}
{\em The following holds:
\begin{enumerate}
    \item If $L=\{L(v)\}_{v\in V_{C_{5}}}$ is an assignment of lists of size $3$ to  $V_{C_{5}}$, and $\vert \bigcup_{v\in V_{C_{5}}} L(v)\vert \neq 3$, then $C_{5}$ is properly $L$-distinguishable.
    
    \item If $L=\{L(v)\}_{v\in V_{C_{6}}}$ is an assignment of lists of size $4$ to  $V_{C_{6}}$, and $\vert \bigcup_{v\in V_{C_{6}}} L(v)\vert \neq 4$, then $C_{6}$ is properly $L$-distinguishable.
\end{enumerate}
}
\end{prop}

\begin{proof}
Fix $k\in\{5,6\}$. Let $v_{0},...,v_{k-1}$ be an enumeration of $V_{C_{k}}$ in clockwise order. We define a proper distinguishing coloring $f$ of $C_{k}$ 
in each case so that $f(v_{i})\in L(v_{i})$ for all $0\leq i\leq k-1$.

(1). For each $0\leq i\leq 2$, pick $c_i\in L(v_i)$ such that $c_{0}, c_{1},$ and $c_{2}$ are pairwise distinct. Let, 
$f(v_{i})=c_{i}$, $f(v_3)\in L(v_3)\setminus \{c_1,c_2\}$, and $f(v_4)\in L(v_4)\setminus\{c_0, f(v_3)\}$.     

(2). For each $0\leq i\leq 3$, pick $c_i\in L(v_i)$ such that $c_{0}, c_{1}, c_{2},$ and $c_{3}$ are pairwise distinct. Let,
   $f(v_{i})=c_{i}$, $f(v_4)\in L(v_4)\setminus \{c_1,c_2,c_3\}$, and $f(v_5)\in L(v_5)\setminus\{c_0, f(v_4), c_3\}$.  
\end{proof}

\begin{prop}\label{Proposition 2.4}
{\em Fix $n\geq 3$. Let $L=\{L(v)\}_{v\in V_{C_{2n+1}}}$ be an assignment of lists of size $3$ to $V_{C_{2n+1}}$, and $\vert \bigcup_{v\in V_{C_{2n+1}}} L(v)\vert \neq 3$, then $C_{2n+1}$ is properly $L$-distinguishable.}
\end{prop}

\begin{proof}
Since $C_{2n+1}$ is connected, there exist two vertices $x$ and $y_{1}$ in $C_{2n+1}$ such that $L(x)\neq L(y_{1})$ and $\{x,y_{1}\}\in E_{C_{2n+1}}$. We color $C_{2n+1}$ as follows:

\begin{enumerate}
\item[$(a)$] Pick any $c_{x}\in L(x)\backslash L(y_{1})$. Color $x$ with $c_{x}$. Let $N(x)$ be the set of vertices adjacent to $x$.
\item[$(b)$] Let $w_{1}\in N(x)\backslash \{y_{1}\}$. Color $w_{1}$ with any color $c_{w_{1}}$ from $L(w_{1})\backslash\{c_{x}\}$. 
\item[$(c)$] Color $y_{1}$ with any color $c_{y_{1}}$ from $L(y_{1})\backslash\{c_{w_{1}}\}$. We note that $c_{x}\not\in L(y_{1})$. Moreover, we can see that $c_{x}, c_{y_{1}}$, and $c_{w_{1}}$ are pairwise distinct.
\item[$(d)$] Let $\{y_{2},...,y_{n}, w_{n},...,w_{2}\}$ be the vertices in the path of length $2n-1$ joining $y_{1}$ and $w_{1}$ in clockwise order (see Figure \ref{Figure 2}).
\item[$(e)$] For any $2\leq k\leq n$, we color $y_{k}$ inductively with any color $c_{y_{k}}$ from $L(y_{k})\backslash\{c_{x}, c_{y_{k-1}}\}$ and color $w_{k}$ inductively with any color $c_{w_{k}}$ from $L(w_{k})\backslash\{c_{x}, c_{w_{k-1}}\}$.
\end{enumerate}
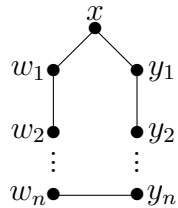
\begin{figure}[!ht]
\centering
\begin{minipage}{\textwidth}
\centering
\begin{tikzpicture}[scale=0.55]
\draw[black,] (0,-1) -- (-1,-2);
\draw[black,] (0,-1) -- (1,-2);
\draw[black,] (-1,-2) -- (-1,-3.5);
\draw[black,] (1,-2) -- (1,-3.5);
\draw[black,] (-1,-5) -- (1,-5);

\draw (0,-1) node {$\bullet$};
\draw (0,-0.6) node {$x$};
\draw (1,-2) node {$\bullet$};
\draw (1.6,-2) node {$y_{1}$};
\draw (1,-3.5) node {$\bullet$};
\draw (1.6,-3.5) node {$y_{2}$};

\draw (1,-4) node {$.$};
\draw (1,-4.2) node {$.$};
\draw (1,-4.4) node {$.$};

\draw (1,-5) node {$\bullet$};
\draw (1.6,-5) node {$y_{n}$};
\draw (-1,-5) node {$\bullet$};
\draw (-1.6,-5) node {$w_{n}$};

\draw (-1,-4) node {$.$};
\draw (-1,-4.2) node {$.$};
\draw (-1,-4.4) node {$.$};

\draw (-1,-3.5) node {$\bullet$};
\draw (-1.6,-3.5) node {$w_{2}$};
\draw (-1,-2) node {$\bullet$};
\draw (-1.6,-2) node {$w_{1}$};

\end{tikzpicture}
\end{minipage}
\caption{\em The cycle graph $C_{2n+1}$.}
\label{Figure 2}
\end{figure}
\textsc{Case 1:} If $c_{w_{n}}\neq c_{y_{n}}$, then $x$ is the only vertex to get color $c_{x}$. Since any nontrivial color-preserving automorphism fixes $x$, $w_{1}$ must map to $y_{1}$, but $c_{y_{1}}\neq c_{w_{1}}$. Thus, $C_{2n+1}$ is properly $L$-distinguishable.

\textsc{Case 2:} If $c_{w_{n}}= c_{y_{n}}$, then we recolor $w_{n}$ with any color $d_{w_{n}}$ from $L(w_{n})\backslash\{c_{y_{n}},c_{w_{n-1}}\}$. If $d_{w_{n}}\neq c_{x}$, then by the arguments of \textsc{Case 1}, we are done. If $d_{w_{n}}=c_{x}$, then $w_{n}$ is the only vertex, other than $x$, to get the color $c_{x}$. Fix any nontrivial color-preserving automorphism $\phi$. Then $\phi$ should either fix both $x$ and $w_{n}$, or map $x$ to $w_{n}$ and $w_{n}$ to $x$.
Similar to \textsc{Case 1}, $\phi$ cannot fix $x$. If $\phi$ maps $x$ to $w_{n}$, then for some $1\leq m\leq n$, we have either $\{c_{w_{m}},c_{y_{m}}\}=\{c_{w_{m}},c_{y_{m+1}}\}$ or $\{c_{w_{m}},c_{y_{m}}\}=\{c_{w_{m-1}},c_{y_{m}}\}$, which is a contradiction as $\{y_{m},y_{m+1}\}$ and $\{w_{m},w_{m-1}\}$ are edges of $C_{2n+1}$.   
\end{proof}

\begin{prop}\label{Proposition 2.5}
{\em Fix $n\geq 4$. Let $L=\{L(v)\}_{v\in V_{C_{2n}}}$ be an assignment of lists of size $3$ to $V_{C_{2n}}$, and $\vert \bigcup_{v\in V_{C_{2n}}} L(v)\vert \neq 3$, then $C_{2n}$ is properly $L$-distinguishable.}
\end{prop}

\begin{proof}
Since $C_{2n}$ is connected, assume $x$, $y_{1}$, $w_{1}$, $c_{x}$, $c_{y_{1}}$, and $c_{w_{1}}$ as in the proof of Proposition \ref{Proposition 2.4}. The following facts will be useful for our proof:

\begin{enumerate}
    \item[$(a)$] $c_{x}\not\in L(y_{1})$,
    \item[$(b)$] $c_{x}, c_{y_{1}}$, and $c_{w_{1}}$ are pairwise distinct.
\end{enumerate}

Let $\{y_{2},...,y_{n-1}, z, w_{n-1},...,w_{2}\}$ be the vertices in the path of length $2n-2$ joining $y_{1}$ and $w_{1}$ in a clockwise order (see Figure \ref{Figure 3}). For $2\leq k\leq n-1$, assume $c_{y_{k}}$ and $c_{w_{k}}$ as in the proof of Proposition \ref{Proposition 2.4}.
Let $L'(z)= L(z)\setminus \{c_x,c_{y_{n-1}}, c_{w_{n-1}}\}$.

\begin{figure}[!ht]
\centering
\begin{minipage}{\textwidth}
\centering
\begin{tikzpicture}[scale=0.55]
\draw[black,] (0,-1) -- (-1,-2);
\draw[black,] (0,-1) -- (1,-2);
\draw[black,] (-1,-2) -- (-1,-3.5);
\draw[black,] (1,-2) -- (1,-3.5);
\draw[black,] (-1,-5) -- (0,-6);
\draw[black,] (1,-5) -- (0,-6);

\draw (0,-1) node {$\bullet$};
\draw (0,-0.6) node {$x$};
\draw (1,-2) node {$\bullet$};
\draw (1.5,-2) node {$y_{1}$};
\draw (1,-3.5) node {$\bullet$};
\draw (1.5,-3.5) node {$y_{2}$};

\draw (1,-4) node {$.$};
\draw (1,-4.2) node {$.$};
\draw (1,-4.4) node {$.$};

\draw (1,-5) node {$\bullet$};
\draw (2,-5) node {$y_{n-1}$};
\draw (0,-6) node {$\bullet$};
\draw (0,-6.4) node {$z$};
\draw (-1,-5) node {$\bullet$};
\draw (-2,-5) node {$w_{n-1}$};

\draw (-1,-4) node {$.$};
\draw (-1,-4.2) node {$.$};
\draw (-1,-4.4) node {$.$};

\draw (-1,-3.5) node {$\bullet$};
\draw (-1.9,-3.5) node {$w_{2}$};
\draw (-1,-2) node {$\bullet$};
\draw (-1.9,-2) node {$w_{1}$};

\end{tikzpicture}
\end{minipage}
\caption{\em The cycle graph $C_{2n}$.}
\label{Figure 3}
\end{figure}
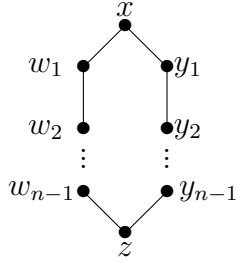

\textsc{Case 1:} Assume $L'(z)\neq \emptyset$. Then assigning any color from $L'(z)$ to $z$ will result in a proper distinguishing coloring since the only vertex colored by $c_x$ is $x$ and $c_{w_{1}}\neq c_{y_{1}}$. 

\textsc{Case 2:} Let $L'(z)= \emptyset$. In order to keep the coloring proper we must assign color $c_x$ to $z$.
Thus, only $x$ and $z$ are colored with $c_{x}$.
Let $f$ be the coloring of $C_{2n}$ obtained thus far.
Without loss of generality, we may assume that there is a nontrivial automorphism $\phi$ preserving $f$.  
Then $\phi$ should either fix both $x$ and $z$, or map $x$ to $z$ and $z$ to $x$.
Similar to \textsc{Case 1}, $\phi$ cannot fix $x$. Thus, $\phi$ maps $x$ to $z$ and $z$ to $x$.

\textsc{Subcase 2.1:} Assume $\phi(y_{1})=y_{n-1}$. Then we must have
$c_{y_{1}} = c_{y_{n-1}}$.
Define $f'$ as follows:
\begin{align}\tag{$\ast$}
    f'(v) = \begin{cases}
        f(v) & \text{ if } v\neq y_{1},\\
        c & \text{ where } c\in L(y_{1})\setminus\{c_{y_{1}}, c_{y_{2}}\}.
    \end{cases}
\end{align}
Since $c_{x}\not\in L(y_{1})$ by $(a)$, we have that $c\neq c_{x}$. Thus, $f'$ is a proper coloring.
The following claim states that $f'$ is a proper distinguishing coloring.

\begin{claim}\label{Claim 2.6}
{\em The only automorphism preserving $f'$ is the identity.}
\end{claim}
\begin{proof}
Let $\psi$ be a nontrivial automorphism preserving $f'$. Similar to \textsc{Case 1}, $\psi$ cannot fix $x$.
Assume that $\psi$ maps $x$ to $z$ and $z$ to $x$. Since $c \neq c_{y_{1}}=c_{y_{n-1}}$, we have $\psi(y_{1})=w_{n-1}$. Consequently, $\psi(w_{1})=y_{n-1}$. Thus,
$c_{y_{1}}= c_{y_{n-1}} = f'(y_{n-1})= f'(w_{1}) = f(w_{1})= c_{w_{1}}$,
which contradicts $(b)$.
\end{proof}

\textsc{Subcase 2.2:} If $\phi(y_{1})=w_{n-1}$, then an argument analogous to \textsc{Subcase 2.1} applies.
\end{proof}

\begin{thm}\label{Theorem 2.7}
    {\em The entries in the following table are correct.}
    
\begin{center}
      \begin{tabular}{|ll|l|l|}
        \hline
             & Graph & $\chi_{D_L}(G)$ \\
             \hline\hline
             & $P_{2t}, t\geq 1$ & $2$\\
             & $P_{2t+1}, t\geq 1$ & $3$ \\
             & $C_4$ & $4$ \\
             & $C_5$ & $3$ \\
             & $C_6$ & $4$\\
             & $C_{2m+1}, m\geq 3$  & $3$ \\
             & $C_{2n}, n\geq 4$ & $3$ \\
               
        \hline
        \end{tabular}
        \\
       Table 1. 
        \label{tab:my_label}
    \end{center}  
\end{thm}

\begin{proof}
In each case, we may assume that the lists are not identical; otherwise, we can color the vertices in a manner identical to the proper distinguishing coloring and apply Fact \ref{Fact:2.1}(4). If the lists are non-identical, then by Propositions \ref{Proposition 2.2}, \ref{Proposition 2.3}, \ref{Proposition 2.4}, \ref{Proposition 2.5}, and Facts \ref{Fact:2.1}(3,4), we are done.
\end{proof}

\section{Brooks' type upper bounds}

Imrich et al. \cite[Theorem 3]{IKPS2017} proved that if $H$ is a connected infinite graph with a finite maximum degree $\Delta(H)$, then $\chi_{D}(H)\leq 2\Delta(H)-1$.
We apply Theorem \ref{Theorem 2.7} and modify the algorithm of \cite[Theorem 3]{IKPS2017} suitably to prove the following upper bound for $\chi_{D_{L}}(G)$ if $G$ is a finite graph:

\begin{thm}\label{Theorem 3.1}
{\em Let $G$ be a connected finite graph with maximum degree $\Delta(G)$. Then $\chi_{D_{L}}(G)\leq 2\Delta(G)-1$, unless $G$ is  $K_{\Delta(G),\Delta(G)}$ or $C_6$. In these cases, $\chi_{D_{L}}(G) = 2\Delta(G)$}.    
\end{thm}

\begin{proof}
   If $\Delta(G)=2$, then by Theorem \ref{Theorem 2.7}, we have $\chi_{D_L}(G) = 4$ if and only if $G$ is $C_4$ or $C_6$. It is easy to see that $\chi_{D_L}(K_{\Delta(G),\Delta(G)}) = 2\Delta(G)$, and $K_{2,2}$ is $C_4$. Hence, we may assume that $\Delta(G) \geq 3$ and $G$ is not $K_{\Delta(G),\Delta(G)}$. We show that in this case $\chi_{D_L}(G) \leq 2\Delta(G)-1$. Let $L=\{L(v)\}_{v\in V_G}$ be an assignment of lists with $|L(z)| = 2\Delta(G)-1$, for all $z\in V_G$. 
   Let $v\in V_G$ be a vertex with degree $\Delta(G)$ and $T$ be a breadth-first search ($BFS$) spanning tree of $G$ rooted at $v$. We use the notation $<$ to denote the $BFS$ order. Also, for $z\in V_{G}$, by $N(z)$ we denote the set of its neighbors in $G$ and $S(z)$ denotes the set of its siblings in $T$. We will define partial functions $f_x: V_{G} \to \bigcup_{z\in V_{G}}L(z)$ for $x\in V_{G}$ such that the following holds: 
\begin{enumerate}
    \item[(i)] if $x < y$, then $f_x\subseteq f_y$,
    \item[(ii)] if $f=\bigcup_{x\in V_{G}} f_x$, then $f(v)\in L(v)$ for all $v\in V_{G}$,
    
    \item[(iii)] the domain of $f_{x}$ includes all vertices till $x$ in the $BFS$ order.
\end{enumerate}
For the base case, we assign pairwise different colors to $v$ and to members of $N(v) = \{v_1,\dots, v_{\Delta(G)}\}$ from the respective lists, say $c_v\in L(v)$ and $c_{v_i}\in L(v_i)$, where $1\leq i\leq \Delta(G)$. 
For the rest, we try to keep $v$ to be the only vertex of $G$ with the property of being colored with $c_v$, while all its neighbors are colored by $c_{v_1},\dots, c_{v_{\Delta(G)}}$. We will refer to this property as ($\ast$). 

We proceed by coloring the vertices of $G$ in their $BFS$ order inductively, as follows: Let $x\in V_{G}$ be the $<$-minimal element for which no color has been assigned and $z<x$ be the immediate $BFS$ predecessor of $x$. We define $f_x$. Let us write,
$$
    S_<(x) = \{y\in S(x): y<x\} \text{ and } N_<(x) = \{y\in N(x): y<x\}.
$$
Then $f_z$ has already assigned colors to each member of $S_<(x) \cup N_<(x)$. Define,
$$
A_x= L(x)\setminus \Big(\{c_v\} \cup f_z[S_<(x)] \cup f_z[N_<(x)]\Big).
$$
If $A_{x} \neq \emptyset$, then let $f_x = f_z \cup \{\langle x, c \rangle\} $,
where $c\in A_x$. If $A_x=\emptyset$, then since $|N(x)| \leq \Delta(G)$ and $|S(x)|\leq \Delta(G)-2$, we must have used $|S_<(x)| + |N_<(x)| = 2\Delta(G) -2$ many colors from $L(x)$ to color the neighbors and siblings that come before $x$. Thus, the following holds: 

\begin{enumerate}
    \item[$(a)$] $c_v\in L(x)$,
    \item[$(b)$] $S_<(x)=S(x)$ and $N_<(x)=N(x)$,
    \item[$(c)$] $|S(x)| = \Delta(G) -2$ and $|N(x)| = \Delta(G)$,
    \item[$(d)$] $N(x) \cap S(x) = \emptyset$,
    \item[$(e)$] for all $y\in S(x) \cup N(x)$, we have $f_z(y)\in L(x)$.
\end{enumerate}
In this case, let $f_x = f_z\cup \{\langle x, c_v \rangle\}$.
In either case, $f_x$ has (i) and (iii). 
Finally, set $f=\bigcup_{x\in V_{G}}f_x$. It is clear that $f$ satisfies (ii), and whenever $f(x)=c_v$ where $x\neq v$, then $(a)-(e)$ must hold. Moreover, if no vertex other than $v$ has property $(\ast)$, then it is easy to see that $f$ is a proper distinguishing coloring. 
However, suppose there is a vertex other than $v$, say $x$, with property $(\ast)$ and is the $<$-minimal such vertex. Then $(a)-(e)$ are satisfied for $x$. We modify $f$ in such a way that $x$ has no longer $(\ast)$, while it remains a proper list coloring.
By a careful introspection we will analyze the following cases: \footnote{We note that \textsc{Case 1} and {\em Case A} of \textsc{Subcase 3.1} arise when the lists are nonidentical.}

\noindent \textsc{Case 1:} If there is $w\in S(x)$ such that $|N(w)| \neq |N(x)|$, then color $x$ with $f(w)$,
which can be done by $(e)$. Then there is no color-preserving automorphism of $G$, which fixes $v$ and maps $x$ to $w$. Meanwhile, by $(d)$, this keeps the coloring proper, and $x$ has no longer the property $(\ast)$.

\noindent \textsc{Case 2:} For every $z\in S(x)$, we have $|N(z)| = |N(x)|$, and there is a $w\in S(x)$ such that $N(w)\neq N(x)$. By $(c)$, fix $y\in N(x)$ such that $y\not\in N(w)$. Since the neighbors of $x$ must have occurred before $x$ in the $BFS$ ordering, we have $f(w), f(y)\in L(x)$ by $(e)$.
We color $x$ with $f(w)$. The rest follows the arguments of \textsc{Case 1}.

\noindent  \textsc{Case 3: } For every $w\in S(x)$, we have $N(w) = N(x)$.
\begin{enumerate}
    \item[]\textsc{Subcase 3.1:} There is a vertex $z\in N(x)$ with no sibling or parent colored with $c_v$. We need to consider two different cases.
    \begin{enumerate}
    \item[]\textit{Case A:}
    If $c_v\not\in L(z)$, then by $(c)$ at most $(\Delta(G)-1) + (\Delta(G)-2) = 2\Delta(G)-3$ number of colors can appear in the list $L(z)$ that are used for coloring $N(z)\cup S(z)$. Hence, there must be $c\in L(z)$ such that $c\neq f(z)$, which also differs from all the colors assigned to $N(z)\cup S(z)$. 
    We color $z$ with $c$ and $x$ with $f(z)$.
    Then $x$ no longer has $(\ast)$ (as no member of $N(v)$ has color $c$), and this modification keeps the coloring proper.
    Moreover, $z$ does not have the property 
    $(\ast)$ as $z$ is colored by $c$, but $c\neq c_{v}$, since $c\in L(z)$ and $c_{v}\not\in L(z)$.
   \item[]  \textit{Case B:} If $c_v\in L(z)$, then color $x$ with $f(z)$ and $z$ with $c_{v}$, which can be done by $(e)$. Clearly, $x$ does not have $(\ast)$, and $z$ does not have $(\ast)$ either by $(b)-(d)$. 
    \end{enumerate}
   
    \item[] \textsc{Subcase 3.2:} Each $y \in N(x)$ has a sibling or parent colored $c_v$. By $(b)$, each $y\in N(x)$ and its siblings must have come before $x$ in the $BFS$ order. Hence, their parent cannot have ($\ast$), unless it is the root $v$. There must exist a $z\in N(x)$ such that $z$ is not a child of $v$. Otherwise, $G$ is $K_{\Delta(G),\Delta(G)}$, which contradicts the assumption that $G$ is not $K_{\Delta(G),\Delta(G)}$. 
    Fix such $z$.
    Let $B(z)$ denotes the set of colors assigned to the vertices of $N(z)\cup S(z)$. We claim that the set $L'(z)=L(z)\backslash \{f(z)\cup B(z)\}$ is non-empty. Let $p$ be the parent of $z$ in $T$. Then $N(z) = S(x) \cup \{p\}\cup \{x\}$, $\vert S(z)\vert \leq \Delta(G)-2$, and $\vert N(z)\vert = \vert S(x)\cup \{p\}\cup \{x\}\vert=\Delta(G)$. The color $c_v$ was used once for $x$ and once among the vertices of $S(z) \cup \{p\}$. Thus, $\vert B(z)\vert\leq 2\Delta(G)-3$. Since $\vert L(z)\vert= 2\Delta(G)-1$, the set $L'(z)$ is non-empty.
    Pick $c\in L'(z)$.
    Using $(e)$, we can color $x$ with $f(z)$ and $z$ with $c$.  We note that $z$ does not get property $(\ast)$. This is because $c\neq c_{v}$ since $c_{v} \in B(z)$.
    Similarly, as in the previous cases, the new coloring remains a proper list coloring, and $v$ is the only $<$-predecessor of $x$ having $(\ast)$, and all of them are fixed as long as $v$ is fixed.
\end{enumerate}
By iterating this procedure, we can recolor every node with property $(\ast)$, and hence we obtain a list-distinguishing proper coloring of $G$.
\end{proof}

\begin{remark}\label{Remark 3.2}
Alikhani--Soltani \cite[Theorem 3.2]{AS2016} proved that $\chi_{D}(G)\leq \Delta(G)+1$ if $\Delta(G)\geq 3$, where $G$ is a connected unicyclic graph. Inspired by that result, we prove that if $G$ is a connected unicyclic graph of girth at least 7, then $\chi_{D_{L}}(G)\leq \Delta(G)$ if $\Delta(G)\geq 3$. First, we observe that, similar to \cite[Lemma 3.2]{CT2006} due to Collins and Trenk, one can prove the following lemma:
\begin{lem}\label{lemma 3.3}
{\em Let $(T,z)$ be a rooted tree, and let $L=\{L(v)\}_{v\in V_{T}}$ be an assignment of lists. A coloring of $(T,z)$ in which each vertex, colored from its lists, that is colored differently from its siblings and from its parent, is a properly $L$-distinguishing coloring.}
\end{lem}
Let $C$ be the unique cycle in $G$ with a set $\{x_{i}:1\leq i\leq t\}$ of vertices. Then, $G \backslash E_{C}$ is the union of trees $T_{x_{i}}$, $1 \leq i \leq t$, where $T_{x_{i}}$ has only one common vertex $x_{i}$ with the cycle $C$. 
Assign a list assignment $L=\{L(v)\}_{v\in V_{G}}$ such that $\vert L(v)\vert= \Delta(G)$ for all $v\in V_{G}$.  
\begin{itemize}
    \item By Theorem \ref{Theorem 2.7}, we have $\chi_{D_{L}}(C)= 3$ since $t\geq 7$.
    \item Define $L^{1}(v):=\mathrm{init}_{3}(L(v))$ for each vertex $v\in V_{G}$, where init$_{3}(L(v))$ denotes the first $3$ elements of $L(v)$, and let $L^{1}=\{L^{1}(v)\}_{v\in V_{C}}$. Since $C$ is properly $L^{1}$-distinguishable (say by coloring $f^1$), we define $f\restriction V_{C}:=f^{1}$.
\end{itemize} 

Fix $1 \leq i \leq t$. 
Working much in the same way as in Alikhani–Soltani \cite[Theorem 3.2]{AS2016} and by applying Lemma \ref{lemma 3.3}, we have $\chi_{D_{L}}(T_{x_{i}}\backslash \{x_{i}\})\leq \Delta(G)$. 
Since $\vert L(v)\vert= \Delta(G)$ for all $v\in V_{G}$, color the vertices of $T_{x_{i}}\backslash\{x_{i}\}$ uniquely from its list (say by $f^i$). 
Define $f\restriction V_{T_{x_{i}}\backslash\{x_{i}\}}:=f^{i}$.

If $\phi$ is a nontrivial automorphism preserving $f$, then all vertices of $C$ are fixed by $\phi$, and consequently all vertices in the $T_{x_{i}}$'s are fixed. Thus, $f$ is a proper distinguishing coloring, and we conclude that $\chi_{D_{L}}(G) \leq \Delta(G)$.
\end{remark}

\section{An application of Lov\'{a}sz's lemma}
For a subset $X\subseteq V_{G}$ of vertices of a graph $G$, let $G[X]$ denote the subgraph of $G$ induced by $X$. 

\begin{lem}\label{Lemma 4.1}{\em (Lov\'{a}sz; \cite{Lov1966})}
{\em If $\sum_{i=1}^{t} x_{i}\geq \Delta(G)+1-t$, then there is a partition of $V_{G}$ into $t$-sets $V_{1},...,V_{t}$ such that $\Delta(G[V_{i}])\leq x_{i}$ for all $1\leq i\leq t$.}     
\end{lem}

\begin{defn}\label{Definition 4.2}
A graph is {\em $k$-connected} if it has at least $k+1$ vertices and the removal of $k-1$ or fewer vertices leaves a connected graph. We denote init$_{k}(L)$ by the first $k$ elements of a list $L$.
\end{defn}

We apply Lemma \ref{Lemma 4.1} and Theorem \ref{Theorem 3.1} to prove the following:

\begin{thm}\label{Theorem 4.3}{\em Fix $7\leq r\leq \Delta(G)+1$. Let $G$ be a $(n-r+1)$-connected graph of order $n$ such that $G$ does not contain a complete bipartite graph $K_{r-1,r-1}$.
Then, $\chi_{D_{L}}(G)\leq 2\Delta(G)-  (3\lfloor\frac{(\Delta(G)+1)}{r}\rfloor-2)$.
}
\end{thm}

\begin{proof}
Let $t=\lfloor\frac{(\Delta(G)+1)}{r}\rfloor$, $h_{i}=r-1$ for all $1\leq i\leq t-1$, and $h_{t}=\Delta(G)-r(t-1)$. Clearly,
\[t\geq 1 \text{ and } \sum_{i=1}^{t} h_{i}=\Delta(G)+1-t.\] 
By Lemma \ref{Lemma 4.1}, there is a partition of $V_{G}$ into $V_{1},...,V_{t}$ such that $\Delta(G[V_{i}])\leq h_{i}=r-1$, for all $1\leq i\leq t-1$ and $\Delta(G[V_{t}])\leq h_{t}=\Delta(G)-r(t-1)$.

\begin{claim}\label{claim 4.4}
   $\chi_{D_{L}}(G)\leq \sum_{i=1}^{t} \chi_{D_{L}}(G[V_{i}])$. 
\end{claim}

\begin{proof}
Let $k_{i}=\chi_{D_{L}}(G[V_{i}])$ for each $1\leq i\leq t$. Assign any list $L(v)$ to each vertex $v \in V_{G}$ such that $\vert L(v)\vert= \sum_{i=1}^{t}k_{i}$. We define a coloring $f$ for the vertices of $G$ as follows:

\begin{enumerate}
    \item Let $L^{1}(v):=$ init$_{k_{1}}(L(v))$ for each vertex $v\in V_{G}$. 
    \vspace{2mm}
    
    Let $L^{1}=\{L^{1}(v)\}_{v\in V_{G[V_{1}]}}$. Since $G[V_{1}]$ is properly $L^{1}$-distinguishable (say by coloring $f^1$), we define $f\restriction V_{G[V_{1}]}:=f^{1}$.
    \vspace{2mm}
    
    \item For any $1<i\leq t$, let $L^{i}(v):=$ init$_{k_{i}}(L(v)\backslash \sum_{k= 1}^{i-1} L^{k}(v)) $ for each vertex $v\in V_{G}$. 
    \vspace{2mm}
    
    Let $L^{i}=\{L^{i}(v)\}_{v\in V_{G[V_{i}]}}$. Since $G[V_{i}]$ is properly $L^{i}$-distinguishable (say by coloring $f^i$), we define $f\restriction V_{G[V_{i}]}:=f^{i}$.
\end{enumerate}  

If $\phi$ is an automorphism of $G$ preserving the coloring $f$, then the range of $G[V_{i}]$ with respect to $\phi$ is $G[V_{i}]$, for all $1\leq i\leq t$. Thus, $\phi\restriction G[V_{i}]$ is an automorphism of $G[V_{i}]$ preserving $f$. Since the above defined colorings of $G[V_{i}]$'s are distinguishing, we have that $\phi\restriction G[V_{i}]$ is a trivial automorphism for all $1\leq i\leq t$. Thus, $f$ is a distinguishing coloring of $G$. 

We can see that $f$ is a proper coloring of $G$ as well. Pick any $x,y\in V_{G}$ such that $\{x,y\}\in E_{G}$. If $x,y\in V_{G[V_{i}]}$ for some $1\leq i\leq t$, then $f(x)\neq f(y)$, as $f\restriction (V_{G[V_{i}]})$ is a proper coloring. If $x\in V_{G[V_{i}]}$ and $y\in V_{G[V_{j}]}$ such that $i\neq j$, then $f(x)\neq f(y)$, as $f$ is defined in a way such that the sets of colors used to color $G[V_{i}]$ and $G[V_{j}]$ are different.
\end{proof} 

\begin{claim}\label{claim 4.5}
{\em $\chi_{D_{L}}(G[V_{i}])\leq 2r-3$ for any $1\leq i\leq t-1$ and $\chi_{D_{L}}(G[V_{t}])\leq 2(\Delta(G)-r(t-1))-1$.}
\end{claim}

\begin{proof}
Fix $1\leq i\leq t-1$. We will analyze the following cases: 

\noindent \textsc{Case 1:}  Suppose $\vert V_{G[V_{i}]}\vert< r$. As coloring the vertices with distinct colors from their respective lists of size $\vert V_{G[V_{i}]}\vert$ yields a proper distinguishing coloring and $r\geq 7$, we have
\[\chi_{D_{L}}(G[V_{i}])\leq \vert V_{G[V_{i}]}\vert \leq r-1\leq 2r-3.\] 
\noindent \textsc{Case 2:} Suppose $\vert V_{G[V_{i}]}\vert\geq r$. Since $G$ is $(n-r+1)$-connected, we have that $G[V_{i}]$ is connected. We note that $\Delta(G[V_{i}])\leq r-1$.
\vspace{2mm}
\begin{enumerate}
    \item[] \textsc{Subcase 2.1:} Let $\Delta(G[V_{i}])< r-1$.
Then, by Theorem \ref{Theorem 3.1}, we have
\[\chi_{D_{L}}(G[V_{i}])\leq 2(\Delta (G[V_{i}]))\leq 2(r-2)=2r-4.\]
    \item[] \textsc{Subcase 2.2:}
    Let $\Delta(G[V_{i}])= r-1$.
Since $G$ contains no $K_{r-1,r-1}$ as subgraphs, neither does $G[V_{i}]$. 
Thus, $G[V_{i}]$ cannot be $K_{\Delta(G[V_{i}]), \Delta(G[V_{i}])}$. 
Moreover, $G[V_{i}]$ cannot be $C_{6}$ as $\vert V_{G[V_{i}]}\vert\geq r\geq 7$.
Thus, by Theorem \ref{Theorem 3.1}, we have
\[\chi_{D_{L}}(G[V_{i}])\leq \vert V_{G[V_{i}]}\vert \leq r-1\leq 2r-3.\]
\end{enumerate}
Similarly, we can see that $\chi_{D_{L}}(G[V_{t}])\leq 2(\Delta(G)-r(t-1))-1$ since $r-1\leq \Delta(G)-r(t-1)$. 
\end{proof}
By applying Claims \ref{claim 4.4} and \ref{claim 4.5}, we get,
\begin{equation*}
\begin{aligned}
 \chi_{D_{L}}(G) & \leq \sum_{i=1}^{t} \chi_{D_{L}}(G[V_{i}])\\
      & \leq (t-1)(2r-3) + 2(\Delta(G)-r(t-1))-1\\
      & \leq 2\Delta(G)- (3t-2).
\end{aligned}
\end{equation*}
\end{proof}

\section{Two upper bounds and special graphs}

Erd\H{o}s and Hajnal \cite{EH1966} introduced 
the coloring number of a graph $G$. The least integer $k$, such that there exists a well-ordering of the vertices of $G$ in which each vertex has fewer than $k$ neighbors that are earlier in the ordering, is defined as the {\em coloring number} of $G$, denoted by $Col(G)$.  
Erd\H{o}s–Rubin–Taylor \cite{ERT1979} and Vizing \cite{Viz1976} independently introduced the list-chromatic number of a graph $G$. Fix an integer $k$. We say that $G$ is {\em $k$-choosable} if for any assignment $L = {L(v)}_{v\in V_{G}}$ of lists of available colors to the vertices of $G$, there is a proper vertex coloring $f$ of $G$ such that $f(v) \in L(v)$ and $\vert L(v)\vert=k$ for all $v\in V_{G}$. The {\em list chromatic number}
of $G$, denoted by $\chi_{L}(G)$, is the minimum integer $k$ such that $G$ is $k$-choosable. The automorphism group of $G$, denoted by $Aut(G)$, is the group consisting of automorphisms of $G$ with composition as the operation. The {\em join} of graphs $G_{1}$, $G_{2}$, ...,$G_{n}$, denoted by $\bigoplus_{1\leq i\leq n} G_{i}$, has
vertex set $\bigcup_{1\leq i\leq n}V_{G_{i}}$ and edge set 
$\bigcup_{1\leq i\leq n}E_{G_{i}}\cup \{xy: x \in V_{G_{i}}, y \in V_{G_{j}}, i \neq j\}$.

\begin{thm}\label{Theorem 5.1}
{\em The following hold:
\begin{enumerate}
    \item For any graph $G=(V_{G},E_{G})$, we have $\chi_{D_{L}}(G)\leq Col(G)D_{L}(G)$.
    \item There exists an asymmetric graph $G$ such that $\chi_{D_{L}}(G)= Col(G)D_{L}(G)$.
    \item If $Aut(G)\cong \Sigma$, where $\Sigma$ is an abelian group of order $p^{m}$, 
    where $p$ is a prime,
    then $\chi_{D_{L}}(G)\leq \chi_{L}(G)+1$. Moreover, the bound $\chi_{D_{L}}(G)\leq \chi_{L}(G)+1$ is sharp. 
    \item If $Aut(G)\cong \Sigma$, where $\Sigma$ is a finite abelian group so that $Aut(G)\cong \prod_{1\leq i\leq k} \mathbb{Z}_{p_{i}^{n_{i}}}$ for some $k$, where $p_{1},...,p_{k}$ are primes not necessarily distinct, then $\chi_{D_{L}}(G)\leq \chi_{L}(G)+k$.
     Moreover, the bound $\chi_{D_{L}}(G)\leq \chi_{L}(G)+k$ is sharp.
\end{enumerate}
}   
\end{thm}

\begin{proof}
(1). Let $c=Col(G)$, $d=D_{L}(G)$, and $l=\{l(v)\}_{v\in V_G}$ be an assignment of lists with $|l(v)| = cd$, for all $v\in V_G$. Let $\{v_{1},...,v_{n}\}$ be an enumeration of $V_{G}$. 
We define an assignment $\{l'(v)\}_{v\in V_{G}}$ of disjoint lists inductively such that $\vert l'(v)\vert = d$ for all $v\in V_G$. 

Let $l'(v_{1})$ be an arbitrary subset of $d$ elements of $l(v_{1})$.
Fix $i\geq 2$. Suppose $l'(v_{1}),..., l'(v_{i-1})$ are defined. We define $l'(v_{i})$. Let $v_{k_1},...,v_{k_r}$ be the $r$ -vertices of $v_{1},...,v_{i-1}$ that are connected to $v_{i}$ by an edge.
Let
    \[X(v_{i}) = l(v_{i})\backslash \bigcup_{1\leq j \leq r} l'(v_{k_{j}}).\]
Since $c>r$, we have
        $\vert X(v_{i}) \vert\geq dc - dr = d(c-r)\geq d.$
Let $l'(v_{i})$ be a $d$-element subset of $X(v_{i})$.

Since $d=D_{L}(G)$, there is a distinguishing coloring $f$ of $G$ such that $f(v)\in l'(v)$ for all $v\in V_{G}$.
By the definition of $\{l'(v)\}_{v\in V_{G}}$, $f$ is a proper coloring as well.

(2). Consider the graph $G$ in Figure \ref{Figure 1}, where $Col(G) = 3$, $D_{L}(G) = 1$, and $\chi_{D_{L}}(G) = 3$.

(3). We slightly modify the methods of Collins, Hovey, and Trenk \cite[Theorem 4.2]{CHT2009}. Let $l=\{l(v)\}_{v\in V_G}$ be an assignment of lists of size
$\chi_{L}(G)+1$. Define $l'(v):=init_{\chi_{L}(G)}(l(v))$ and $l'=\{l'(v)\}_{v\in V_{G}}$. Let $f$ be a proper coloring of $G$ where $f(v) \in l'(v)$ for all $v\in V_{G}$.
Now $\Sigma\cong \mathbb{Z}_{p^{m}}$. If $\sigma$ is a generator of $Aut(G)$, then $\tau=\sigma^{p^{m-1}}$ is a nontrivial automorphism of $G$. Let $\omega$ be a nontrivial automorphism of $G$ of order $p^{t}$ for $t\geq 2$. Clearly, $\tau$ is a power of $\omega$. Since $\tau$ is a nontrivial automorphism, there is a vertex $v$ not fixed by $\tau$. Then $v$ is not fixed by $\omega$ as well. We recolor $v$ by the unique new color from $l(v)\backslash l'(v)$. This new coloring is an $l$-proper distinguishing coloring of $G$.

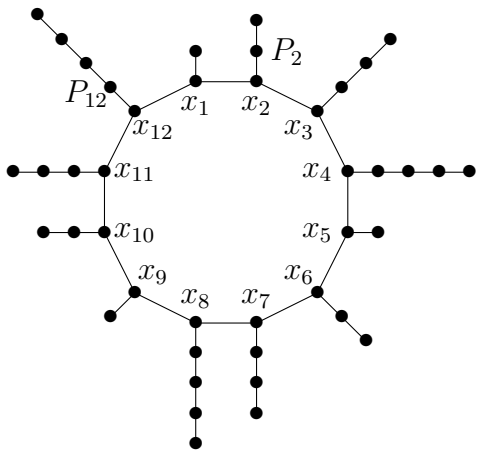
\begin{figure}[!ht]
\centering
\begin{minipage}{\textwidth}
\centering
\begin{tikzpicture}[scale=0.8]
\draw[black,] (-0.5,2) -- (0.5,2);
\draw[black,] (-0.5,2) -- (-0.5,2.5);

\draw[black,] (0.5,2) -- (1.5,1.5);
\draw[black,] (0.5,2) -- (0.5,3);

\draw[black,] (1.5,1.5) -- (2,0.5);
\draw[black,] (1.5,1.5) -- (2.7,2.7);

\draw[black,] (2, 0.5) -- (2, -0.5);
\draw[black,] (2, 0.5) -- (4,0.5);

\draw[black,] (2, -0.5) -- (1.5, -1.5);
\draw[black,] (2, -0.5) -- (2.5,-0.5);

\draw[black,] (1.5, -1.5) -- (0.5, -2);
\draw[black,] (1.5, -1.5) -- (2.3, -2.3);

\draw[black,] (0.5,-2) -- (-0.5,-2);
\draw[black,] (0.5, -2) -- (0.5, -3.5);

\draw[black,] (-0.5,-2) -- (-1.5,-1.5);
\draw[black,] (-0.5, -2) -- (-0.5, -4);

\draw[black,] (-1.5,-1.5) -- (-2,-0.5);
\draw[black,] (-1.5,-1.5) -- (-1.9, -1.9);

\draw[black,] (-2, -0.5) -- (-2, 0.5);
\draw[black,] (-2, -0.5) -- (-3, -0.5);

\draw[black,] (-2, 0.5) -- (-1.5, 1.5);
\draw[black,] (-2, 0.5) -- (-3.5, 0.5);

\draw[black,] (-1.5, 1.5) -- (-0.5, 2);
\draw[black,] (-1.5, 1.5) -- (-3.1, 3.1);

\draw (-0.5,2) node {$\bullet$};
\draw (-0.5,1.6) node {$x_{1}$};
\draw (-0.5,2.5) node {$\bullet$};

\draw (0.5,2) node {$\bullet$};
\draw (0.5,1.6) node {$x_{2}$};
\draw (0.5,2.5) node {$\bullet$};
\draw (1,2.5) node {$P_{2}$};
\draw (0.5,3) node {$\bullet$};

\draw (1.5,1.5) node {$\bullet$};
\draw (1.2,1.2) node {$x_{3}$};
\draw (1.9,1.9) node {$\bullet$};
\draw (2.3,2.3) node {$\bullet$};
\draw (2.7,2.7) node {$\bullet$};

\draw (2,0.5) node {$\bullet$};
\draw (1.5,0.5) node {$x_{4}$};
\draw (2.5,0.5) node {$\bullet$};
\draw (3,0.5) node {$\bullet$};
\draw (3.5,0.5) node {$\bullet$};
\draw (4,0.5) node {$\bullet$};

\draw (2,-0.5) node {$\bullet$};
\draw (1.5,-0.5) node {$x_{5}$};
\draw (2.5,-0.5) node {$\bullet$};

\draw (1.5,-1.5) node {$\bullet$};
\draw (1.2,-1.2) node {$x_{6}$};
\draw (1.9,-1.9) node {$\bullet$};
\draw (2.3,-2.3) node {$\bullet$};

\draw (0.5,-2) node {$\bullet$};
\draw (0.5,-1.6) node {$x_{7}$};
\draw (0.5,-2.5) node {$\bullet$};
\draw (0.5,-3) node {$\bullet$};
\draw (0.5,-3.5) node {$\bullet$};

\draw (-0.5,-2) node {$\bullet$};
\draw (-0.5,-1.6) node {$x_{8}$};
\draw (-0.5,-2.5) node {$\bullet$};
\draw (-0.5,-3) node {$\bullet$};
\draw (-0.5,-3.5) node {$\bullet$};
\draw (-0.5,-4) node {$\bullet$};

\draw (-1.5,-1.5) node {$\bullet$};
\draw (-1.2,-1.2) node {$x_{9}$};
\draw (-1.9,-1.9) node {$\bullet$};

\draw (-2,-0.5) node {$\bullet$};
\draw (-1.5,-0.5) node {$x_{10}$};
\draw (-2.5,-0.5) node {$\bullet$};
\draw (-3,-0.5) node {$\bullet$};

\draw (-2,0.5) node {$\bullet$};
\draw (-1.5,0.5) node {$x_{11}$};
\draw (-2.5,0.5) node {$\bullet$};
\draw (-3,0.5) node {$\bullet$};
\draw (-3.5,0.5) node {$\bullet$};

\draw (-1.5,1.5) node {$\bullet$};
\draw (-1.2,1.2) node {$x_{12}$};
\draw (-1.9,1.9) node {$\bullet$};
\draw (-2.3,2.3) node {$\bullet$};
\draw (-2.3,1.8) node {$P_{12}$};
\draw (-2.7,2.7) node {$\bullet$};
\draw (-3.1,3.1) node {$\bullet$};
\end{tikzpicture}
\end{minipage}
\caption{\em The graph $C'_{12}$, where $Aut(C'_{12})=\mathbb{Z}_{3}$.}
\label{Figure 4}
\end{figure}

For the second assertion, consider the cycle graph $C_{4n}=(V_{C_{4n}}, E_{C_{4n}})$ for any integer $n\geq 2$. Let $x_{1},...,x_{4n}$ be the vertices of $V_{C_{4n}}$ in a clockwise order.
We construct a graph ${C'_{4n}}$ by adding to each $x_k$ a path $P_{k}$ of length $r$, where $r \in \{1, 2, 3, 4\}$ and $k \equiv r$ (mod 4) (see Figure \ref{Figure 4}).
Now, $Aut(C_{4n})$ is isomorphic to the dihedral group of order $8n$. The added paths in ${C'_{4n}}$ break all the reflectional symmetries of $C_{4n}$ since any 3 consecutive nodes in $C_{4n}$ have different added paths. However, the added paths of ${C'_{4n}}$ do not break all of the symmetries of $C_{4n}$ due to its periodicity. Since any automorphism of ${C'_{4n}}$ can only send a node $x\in V_{C_{4n}}$ to a node $y\in V_{C_{4n}}$ such that $P_{x}$ and $P_{y}$ are of the same length, ${C'_{4n}}$ keeps $n$ out of the $4n$ rotational symmetries of $C_{4n}$. Thus, $Aut({C'_{4n}})\cong \mathbb{Z}_n$.
 
\begin{claim}\label{Claim 5.2}
    {\em $\chi_L({C'_{4n}}) = 2$ and $\chi_{D_L}({C'_{4n}}) = 3$.}
\end{claim}	
\begin{proof}
Let $H=C'_{4n}$ and
$l=\{l(v)\}_{v\in V_{H}}$ be an assignment of lists of size $2$.
Let $f$ be a proper coloring of $C_{4n}$ such that $f(v) \in l(v)$ for all $v\in V_{C_{4n}}$.\footnote{as if $C$ is an even cycle then the list chromatic number $\chi_{L}(C)$ of $C$ is 2.} 
We extend $f$ to a proper coloring of $H$ in a way such that the adjacent vertices of the $P_{i}$'s get different colors from their respective lists. So, $\chi_L(H) = 2$.
We show $\chi_{D_L}(H)= 3$. Any $2$-proper  coloring $f$ of $H$ must color alternately each of the vertices in both $C_{4n}$ and the added paths. 
So, there is a nontrivial $f$-preserving automorphism of $H$ that maps a node $x\in V_{C_{4n}}$ to a node $y\in V_{C_{4n}}$ so that $P_{x}$ and $P_{y}$ have the same length.
Thus, $\chi_{D_L}(H)\geq \chi_{D}(H)> 2$.
We prove $\chi_{D_L}(H)\leq 3$. 
Let $l=\{l(v)\}_{v\in V_{H}}$ be an assignment of lists of size $3$. 
By Theorem \ref{Theorem 2.7}, $C_{4n}$ is properly $l'$-distinguishable if $l'=\{l(v)\}_{v\in V_{C_{4n}}}$. We extend $f$ to a proper distinguishing coloring of $H$ such that the adjacent vertices of $P_{i}$'s get different colors from their respective lists.
\end{proof}

(4). The bound $\chi_{D_{L}}(G)\leq \chi_{L}(G)+k$ follows from the arguments of \cite[Theorem 4.4]{CHT2009} and the proof of the first assertion of (3). For the second assertion, we prove that given a finite abelian group $\Gamma=\prod_{1\leq i\leq k} \mathbb{Z}_{p_{i}^{n_{i}}}$ where $n_{1},...,n_{k}$ are different positive integers, there exists a graph $H$ such that $Aut(H)=\Gamma$ and $\chi_{D_{L}}(H)=\chi_{L}(H)+k$.
Let $H_{i}=C'_{4p_{i}^{n_{i}}}$ for all $1\leq i\leq k$ where $C'_{4p_{i}^{n_{i}}}$ is the graph constructed in (3). Let $H = \bigoplus_{1\leq i\leq k} H_{i}$.  
Fix any $1\leq i\leq k$. By Claim \ref{Claim 5.2}, we have $\chi_L({H_{i}}) = 2$ and $\chi_{D_L}({H_{i}}) = 3$. We note that $H_{i}$ is triangle free and is not a complete bipartite graph. Moreover, $H_{i}\not\cong H_{j}$ if $i\neq j$. We recall the following lemma.

\begin{lem}\label{Lemma 5.3}{(Collins, Hovey, and Trenk \cite[Lemma 5.3]{CHT2009})}
{\em Suppose each of the graphs $G_{1},...,G_{n}$, is triangle free, and is not a
complete bipartite graph, and also suppose $G_{i} \not\cong G_{j}$ whenever $i \neq j$. Then, $Aut(\bigoplus_{1\leq i\leq n} G_{i}) = \prod_{i=1}^{n} Aut(G_{i})$.}
\end{lem}

By applying the techniques in \cite[Corollary 5.5]{CHT2009} due to Collins, Hovey, and Trenk and the arguments of Claim \ref{claim 4.4}, we can prove the following lemma:

\begin{lem}\label{Lemma 5.4}
{\em If $Aut(\bigoplus_{1\leq i\leq n} G_{i}) = \prod_{i=1}^{n} Aut(G_{i})$, then $\chi_{D_{L}}(\bigoplus_{1\leq i\leq n} G_{i})=\sum_{i=1}^{n}\chi_{D_{L}}(G_{i})$. }
\end{lem}

By Lemma \ref{Lemma 5.3}, $Aut(H)=\prod_{i=1}^{k} \mathbb{Z}_{p_{i}^{n_{i}}}$. We apply Lemma \ref{Lemma 5.4} to the graph $H$ to conclude that
\[\chi_{D_{L}}(H)=\sum_{i=1}^{k} \chi_{D_{L}}(H_{i})=3k \text{, and } \chi_{L}(H)=\sum_{i=1}^{k} \chi_{L}(H_{i})=2k.\]
\end{proof}

\begin{question}
{\em Does there exist a graph $G$ 
such that $\chi_{D_{L}}(G)=Col(G)D_{L}(G)$ and $G$ is not asymmetric?} 
\end{question}

\subsection{Book graphs and Friendship graphs}

\begin{defn}\label{Definition 5.5}
The {\em $n$-book graph $B_{n}$} ($n \geq 2$) (see Figure \ref{Figure 5}) is defined as the Cartesian product of the star graph $K_{1,n}$ and the path graph $P_{2}$.
We call every $C_{4}$ in the book graph $B_{n}$ a page of $B_{n}$. The $n$ pages of $B_{n}$ are denoted by $v_{0}w_{0}v_{i}w_{i}$ for $1\leq i\leq n$. 
The {\em friendship
graph $F_{n}$}  $(n \geq 2)$ is obtained by joining $n$ copies of the cycle graph $C_{3}$ with a
common vertex $w$ (see Figure \ref{Figure 6}).
\end{defn}

Alikhani and Soltani \cite{AS, AS2017} proved that $D_{L}(F_{n})= D(F_{n})=\lceil\frac{1+\sqrt{8n+1}}{2}\rceil$ and $D_{L}(B_{n})= D(B_{n})=\lceil \sqrt{n}\,\rceil$ for any $n\geq 2$. 
Define
\begin{itemize}
    \item $A=\{n: (\lceil \sqrt{n}\,\rceil)^2-(\lceil \sqrt{n}\,\rceil) + 2 \leq n \leq (\lceil \sqrt{n}\,\rceil)^2\}$,
    \item $B= \{n: (\lceil \sqrt{n}\,\rceil-1)^2 < n \leq (\lceil \sqrt{n}\,\rceil)^2-(\lceil \sqrt{n}\,\rceil)+1\}$.
\end{itemize}

The purpose of this section is to illustrate
the fact that the values of $\chi_{D_L}(B_n)$ and $\chi_{D}(B_n)$ depend on whether $n\in A$ or $n\in B$, while the values of $D_L(B_n)$ and $D(B_n)$ have a unique form in each case, independently of the value of $n$ (see \cite{AS, AS2017}).
We prove that the entries in the following table are correct (see Theorem \ref{Theorem 5.11} and Remark \ref{Remark 5.12}).
    
\begin{center}
      \begin{tabular}{|ll|l|l|}
        \hline
             & Graph ($G$) & $\chi_{D}(G)$ &$\chi_{D_L}(G)$\\
             \hline\hline
             & Book graph $(B_{n})$ if $n\in A$ & $2+ \lceil\sqrt{n}\,\rceil$ & $2+ \lceil\sqrt{n}\,\rceil$ \\
             
             & Book graph $(B_{n})$ if $n\in B$ & $1+ \lceil\sqrt{n}\,\rceil$ & $1+ \lceil\sqrt{n}\,\rceil$ \\\
             
             & Friendship graph $(F_{n})$  &$1+\lceil\frac{1+\sqrt{8n+1}}{2}\rceil$ & $1+\lceil\frac{1+\sqrt{8n+1}}{2}\rceil$\\
             
        \hline
        \end{tabular}
        \\
       Table 2. 
        \label{tab:my_label}
    \end{center}  

\begin{observation}\label{Observation 5.6}
    For $n\geq 2$, let $f$ be a coloring of the set of vertices of $B_n$. Then $f$ is a proper distinguishing coloring of $B_n$ if and only if the following holds:
    \begin{enumerate}
        \item[$(a)$] $(f(v_i), f(w_i))\neq (f(v_j), f(w_j)) $ for $0< i\neq j\leq n$,
        \item[$(b)$] $f(v_i)\neq f(w_i)$ for $i\leq n$,
        \item[$(c)$] $f(v_i)\neq f(v_0)$ and $f(w_i)\neq f(w_0)$ for $0<i\leq n$. 
    \end{enumerate}
    Condition $(a)$ makes $f$ distinguishing, whereas conditions $(b)$ and $(c)$ make it proper.
\end{observation}

\begin{figure}[!ht]\label{Figure 5}
\centering
\begin{minipage}{\textwidth}
\centering
\begin{tikzpicture}[scale=1]
\draw[black,] (0,0) -- (0,-1.5);

\draw[black,] (-1,0.5) -- (-1,-1);
\draw[black,] (1,0.5) -- (1,-1);

\draw[black,] (2,0.5) -- (2,-1);
\draw[black,] (-2,0.5) -- (-2,-1);

\draw[black,] (5,0.5) -- (5,-1);
\draw[black,] (-5,0.5) -- (-5,-1);

\draw[black,] (0,0) -- (1,0.5);
\draw[black,] (0,0) -- (2,0.5);
\draw[black,] (0,0) -- (5,0.5);
\draw[black,] (0,0) -- (-1,0.5);
\draw[black,] (0,0) -- (-2,0.5);
\draw[black,] (0,0) -- (-5,0.5);

\draw[black,] (0,-1.5) -- (1,-1);
\draw[black,] (0,-1.5) -- (2,-1);
\draw[black,] (0,-1.5) -- (5,-1);
\draw[black,] (0,-1.5) -- (-1,-1);
\draw[black,] (0,-1.5) -- (-2,-1);
\draw[black,] (0,-1.5) -- (-5,-1);

\node[circ] at (0,0) {$v_{0}$};
\node[circ] at (0,-1.5) {$w_{0}$};
\node[circ] at (1,0.5) {$v_{2}$};
\node[circ] at (1,-1) {$w_{2}$};
\node[circ] at (2,0.5) {$v_{4}$};
\node[circ] at (2,-1) {$w_{4}$};
\node[circ] at (5,0.5) {$v_{6}$};
\node[circ] at (5,-1) {$w_{6}$};

\node[circ] at (-1,0.5) {$v_{1}$};
\node[circ] at (-1,-1) {$w_{1}$};
\node[circ] at (-2,0.5) {$v_{3}$};
\node[circ] at (-2,-1) {$w_{3}$};
\node[circ] at (-5,0.5) {$v_{5}$};
\node[circ] at (-5,-1) {$w_{5}$};

\end{tikzpicture}
\end{minipage}
\caption{\em Book graph $B_{6}$.}
\label{Figure 5}
\end{figure}
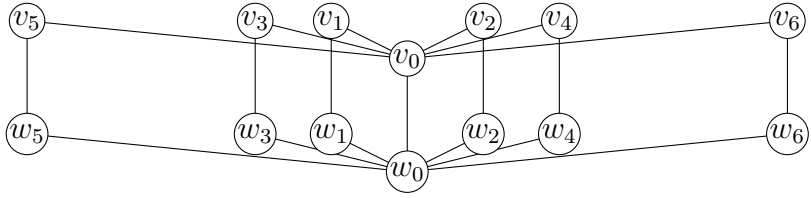

\begin{lem}\label{Lemma 5.7}
 {\em For all $n\geq 2$, $\chi_{D}(B_n) \leq \lceil \sqrt{n}\,\rceil + 2$.}
\end{lem}
\begin{proof}
    Consider a $\lceil \sqrt{n}\,\rceil$-distinguishing coloring $f$ of $B_n$. Using $f$, we define a $(\lceil \sqrt{n}\,\rceil +2)$-proper distinguishing coloring of $B_n$. Let $C= \{1,\dots,\lceil \sqrt{n}\,\rceil+2\}$ be a color set. Since $f$ is distinguishing, for each pair $(v_i,w_i)$ and $(v_j,w_j)$ with $0<i\neq j$, we must have $(f(v_i),f(w_i))\neq (f(v_j),f(w_j))$. Define $f': B_n \to C $ such that $f'(v_0) = \lceil \sqrt{n}\,\rceil+1$, $f'(w_0)= \lceil \sqrt{n}\,\rceil+2$, and
    \begin{align*}
        f'(v_i) &=\begin{cases}
            \lceil \sqrt{n}\,\rceil+2 & \text{ if } f(v_i) = f(w_i),\\
            f(v_i) & \text{ otherwise. }
        \end{cases}    
    \end{align*}
    Clearly, $f'$ is a $(\lceil \sqrt{n}\,\rceil+2)$-proper distinguishing coloring.
\end{proof}
\begin{lem}\label{Lemma 5.8}
   {\em For all $n\geq 2$, $\lceil \sqrt{n}\,\rceil < \chi_{D}(B_n)$.}
\end{lem}
\begin{proof}
    Since $n\geq 2$, we may assume that $\lceil \sqrt{n}\,\rceil = m+1$, for some $0\neq m \in \mathbb{N}$. Suppose we have a proper distinguishing coloring $f$ using only $m+1$ many colors with color set $C=\{1,\dots, m+1\}$. Then conditions $(a)-(c)$ must hold for $f$ from Observation \ref{Observation 5.6}.    
    The total number of possibilities to color $(v_i,w_i)$ is $(m+1)^2$ for any $0<i\leq n$. But $f$ cannot use those pairs $(a,b)\in C\times C$ for which either $a=b$ or $a=f(v_0)$ or $b=f(w_0)$. Thus, the number of color pairs $(a,b)\in C \times C$ that are allowed to color $(v_i,w_i)$ is $(m+1)^2 - 3(m+1) + 3$ (i.e., identical pairs, pairs with the first coordinate $f(v_0)$, and pairs with the second coordinate $f(w_0)$ have to be avoided). But then
$$
     (m+1)^2 - 3(m+1) + 3= m^2 + 1 -m \leq m^2<n.
$$
Consequently, $(a)$ in Observation \ref{Observation 5.6} cannot be satisfied, i.e., there will be pairs $(v_i,w_i)$ and $(v_j,w_j)$ such that $(f(v_i), f(w_i)) = (f(v_j), f(w_j))$.
\end{proof}

\begin{lem}\label{Lemma 5.9}
   {\em Let $B_n$ be such that $\lceil \sqrt{n}\,\rceil = m+1$, for some $0\neq m\in \mathbb{N}$. Then $\chi_{D}(B_n) < \lceil \sqrt{n}\,\rceil + 2$ if and only if $m^2 < n \leq m^2 + m+1$.}
\end{lem}
\begin{proof}
    Clearly, $m^2<n\leq (m+1)^2$. Let $C = \{1,\dots, m+2\}$ be a color set. Fix $a,b\in C$. Color $v_0$ with $a$ and $w_0$ with $b$. The total number of possibilities to color $(v_i,w_i)$ is $(m+2)^2$ for $0<i\leq n$, among which $m^2 + m + 1$ possibilities are allowed (by removing identical pairs, pairs with the first coordinate $a$ and pairs with the second coordinate $b$). If $n\leq m^2 + m + 1$, then we can find a coloring $f$ satisfying $(a)-(c)$ such that $f(v_0) = a$ and $f(w_0) = b$. If $n> m^2 + m + 1$, then following the arguments of Lemma \ref{Lemma 5.8}, we cannot find a proper distinguishing coloring.
\end{proof}

\begin{thm}\label{Theorem 5.10}
    {\em For all $n\geq 2$, we have $\chi_D(B_n) = \lceil \sqrt{n}\,\rceil + 1$ or $\chi_D(B_n) = \lceil \sqrt{n}\,\rceil + 2$. Moreover, if $\lceil \sqrt{n}\,\rceil = m+1$, for some $0\neq m\in \mathbb{N}$, then
    \begin{enumerate}
        \item[(i)] $\chi_D(B_n) = \lceil \sqrt{n}\,\rceil + 1$, whenever $m^2 < n\leq m^2 + m + 1$,
        \item[(ii)]  $\chi_D(B_n) = \lceil \sqrt{n}\,\rceil + 2$, whenever $m^2 +m + 2\leq n \leq (m+1)^2$.
    \end{enumerate}}
\end{thm}
\begin{proof}
     This follows from Lemmas \ref{Lemma 5.7}, \ref{Lemma 5.8}, and \ref{Lemma 5.9}.
\end{proof}

\begin{thm}\label{Theorem 5.11}
  {\em For all $n\geq 2$, $\chi_{D_L}(B_n) = \chi_D(B_n)$.}
\end{thm}
\begin{proof}
   It is enough to show that $\chi_{D_L}(B_n) \leq \chi_D(B_n)$. Let $\lceil \sqrt{n}\,\rceil = m+1$, for some $0\neq m\in \mathbb{N}$. Let $\{L(x_i)\}_{x_i\in V_{B_n}}$ be an assignment of lists with $|L(x_i)| = \chi_{D}(B_n)$, for every $x_i\in B_n$. Fix $c_{v_0}\in L(v_0)$ and $c_{w_0}\in L(w_0)$ so that $c_{v_0}\neq c_{w_0}$. Fix a pair $(v_i,w_i)$, where $0<i\leq n$. By $(L(v_i), L(w_i))$, we denote the number of color pairs $(a_i,b_i)\in L(v_i)\times L(w_i)$ such that $a_i\neq c_{v_0}$, $b_i\neq c_{w_0}$, and $a_{i}\neq b_{i}$.
    
\noindent \textsc{Case 1:} Assume $\chi_D(B_n) =\lceil \sqrt{n}\,\rceil+2$. By Theorem \ref{Theorem 5.10}(ii), we have $n\leq (m+1)^2$. Thus, $|(L(v_i), L(w_i))|\geq  m^2 + 3m + 6 \geq (m+1)^2 \geq n$. 

\noindent \textsc{Case 2:} Assume $\chi_D(B_n) =\lceil \sqrt{n}\,\rceil+1$. Since by Theorem \ref{Theorem 5.10}(i), we have $n \leq m^2 + m + 1$, thus $|(L(v_i), L(w_i))|\geq m^2 + m + 1\geq n$.

\noindent In both cases, there is a proper distinguishing coloring $f$ where $f(v_0) = c_{v_0}$ and $f(w_0)=c_{w_0}$.
\end{proof}

\begin{remark}\label{Remark 5.12}
We remark that 
$\chi_{D_{L}}(F_{n}) = \chi_{D}(F_{n}) = 1+\lceil\frac{1+\sqrt{8n+1}}{2}\rceil$ for every $n\geq 2$.
First, we prove that $\chi_{D}(F_{n})\geq 1+D(F_{n})$. If $L$ is a proper distinguishing coloring for $F_{n}$, and the color of the two vertices on the base of the $i$-th triangle is $x_{i}, y_{i}$ (see Figure \ref{Figure 6}), then the following holds:

\begin{enumerate}
    \item For every $i\in \{1,...,n\}$, $x_{i}\neq y_{i}$.
    \item For every $i,j \in \{1,...,n\}$ where $i\neq j$, $\{x_{i}, y_{i}\}\neq \{x_{j}, y_{j}\}$.
    \item The color of the central vertex $w$, say $z$, cannot be $x_{i}$ or $y_{i}$ for any $i\in \{1,...,n\}$.
\end{enumerate}

Thus, $\chi_{D_{L}}(F_{n})\geq\chi_{D}(F_{n})\geq min\{s:\binom{s}{2}\geq n\}+1=\lceil\frac{1+\sqrt{8n+1}}{2}\rceil+1=D(F_{n})+1$.\footnote{In \cite{AS2017}, Alikhani and Soltani proved that $D(F_{n})=min\{s:\binom{s}{2}\geq n\}=\lceil\frac{1+\sqrt{8n+1}}{2}\rceil$.} 
We show that $\chi_{D_{L}}(F_{n})\leq D(F_{n})+1$. 
Let $L=\{L(v)\}_{v\in V_{F_{n}}}$ be a list assignment to $F_{n}$ such that $\vert L(v)\vert =D(F_{n})+1$. Pick a color for the central vertex $w$, say $c_{w}\in L(w)$. Then $L(i)=L(v_{i})\backslash \{c_{w}\} $ has cardinality $D(F_{n})$ or $D(F_{n})+1$. By the methods of \cite{AS}, we can color $v_{i}$'s and $w_{i}$'s in a distinguishing way. Moreover, the coloring is proper as well. Consequently, $\chi_{D_{L}}(F_{n})=D(F_{n})+1$.  

\begin{figure}[!ht]\label{Figure 6}
\centering
\begin{minipage}{\textwidth}
\centering
\begin{tikzpicture}[scale=0.7]
\draw[black,] (0,0) -- (0.5,2);
\draw[black,] (0,0) -- (-0.5,2);
\draw[black,] (0.5,2) -- (-0.5,2);

\draw[black,] (0,0) -- (2,0.5);
\draw[black,] (0,0) -- (2,-0.5);
\draw[black,] (2, 0.5) -- (2, -0.5);
\draw (2, 0.5) node {$\bullet$};
\draw (2, -0.5) node {$\bullet$};

\draw[black,] (0,0) -- (0.5,-2);
\draw[black,] (0,0) -- (-0.5,-2);
\draw[black,] (0.5,-2) -- (-0.5,-2);
\draw (0.5,-2) node {$\bullet$};
\draw (-0.5,-2) node {$\bullet$};

\draw[black,] (0,0) -- (-2,0.5);
\draw[black,] (0,0) -- (-2,-0.5);
\draw[black,] (-2, 0.5) -- (-2, -0.5);
\draw (-2, 0.5) node {$\bullet$};
\draw (-2, -0.5) node {$\bullet$};

\node[circ] at (0,0) {$w$};
\node[circ] at (-0.5,2) {$v_{1}$};
\draw (-0.5,2.6) node {$x_{1}$};
\node[circ] at (0.5,2) {$v_{2}$};
\draw (0.5,2.6) node {$x_{2}$};
\node[circ] at (2,0.5) {$v_{3}$};
\draw (2.7,0.5) node {$x_{3}$};
\node[circ] at (2,-0.5) {$v_{4}$};
\draw (2.7,-0.5) node {$x_{4}$};
\node[circ] at (0.5,-2) {$v_{5}$};
\draw (0.5,-2.6) node {$x_{5}$};
\node[circ] at (-0.5,-2) {$v_{6}$};
\draw (-0.5,-2.6) node {$x_{6}$};
\node[circ] at (-2,-0.5) {$v_{7}$};
\draw (-2.7,-0.5) node {$x_{7}$};
\node[circ] at (-2,0.5) {$v_{8}$};
\draw (-2.7,0.5) node {$x_{8}$};

\draw (0.5,0.5) node {$z$};
\end{tikzpicture}
\end{minipage}
\caption{\em Coloring (L) of the Friendship graph $F_{4}$ with central vertex $w$.}
\label{Figure 6}
\end{figure}
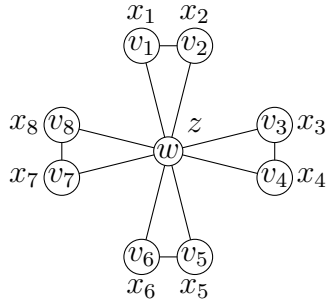

\end{remark}

\section{Acknowledgement} The authors are very thankful to the two anonymous referees for reading the manuscript in detail and for providing suggestions that improved the quality of the paper.

\textbf{Funding} The first author was supported by the EK\"{O}P-24-4-II-ELTE-996 University Excellence scholarship program of the Ministry for Culture and Innovation from the source of the National Research, Development and Innovation fund. The second author was supported by the UNKP-23-3 New National Excellence Program of the Ministry for Culture and Innovation from the source of the National Research, Development and Innovation Fund.

\textbf{Data availability} No data are associated with this article.

\textbf{Declarations}

\textbf{Conflict of interest} The authors declare that they have no conflict of interest.

\end{document}